\documentclass{article}
\usepackage[a4paper]{geometry}
\usepackage{amsmath,amssymb,amsthm,enumitem,environ,color,tikz}
\usepackage[initials]{amsrefs}
    
\newtheorem{theorem}{Theorem}[section]

\theoremstyle{definition}
\newtheorem{example}[theorem]{Example}
\theoremstyle{plain}

\numberwithin{equation}{section}
\numberwithin{figure}{section}

\renewcommand{\geq}{\geqslant}
\renewcommand{\leq}{\leqslant}

\setlist[enumerate]{leftmargin=20pt,itemsep=0pt,topsep=0pt}
\setlist[enumerate,1]{label=\emph{(\roman*)},ref={(\roman*)}}

\makeatletter
\renewcommand\section{\@startsection {section}{1}{\z@}%
                                   {-3.5ex \@plus -1ex \@minus -.2ex}%
                                   {1.3ex \@plus.2ex}%
                                   {\normalfont\large\scshape}}
\makeatother

\title{ \vspace{-5ex}\bf \Large Stability of the Denjoy--Wolff theorem}

\makeatletter
\renewcommand*{\@fnsymbol}[1]{\hspace*{-10pt}}
\makeatother

\author{Argyrios Christodoulou and Ian Short\thanks{School of Mathematics and Statistics, The Open University, Milton Keynes, MK7 6AA, United Kingdom.}}

\date{\vspace{-5ex}}

\begin{document}

\maketitle

\begin{abstract}
The Denjoy--Wolff theorem is a foundational result in complex dynamics, which describes the dynamical behaviour of the sequence of iterates of a holomorphic self-map $f$ of the unit disc $\mathbb{D}$. Far less well understood are nonautonomous dynamical systems $F_n=f_n\circ f_{n-1} \circ \dots \circ f_1$ and $G_n=g_1\circ g_{2} \circ \dots \circ g_n$, for $n=1,2,\dotsc$, where $f_i$ and $g_j$ are holomorphic self-maps of $\mathbb{D}$. Here we obtain a thorough understanding of such systems $(F_n)$ and $(G_n)$ under the assumptions that $f_n\to f$ and $g_n\to f$. We determine when the dynamics of $(F_n)$ and $(G_n)$ mirror that of $(f^n)$, as specified by the Denjoy--Wolff theorem, thereby providing insight into the stability of the Denjoy--Wolff theorem under perturbations of the map $f$. 
\end{abstract}

\section{Introduction}

Fundamental to this paper is the Denjoy--Wolff theorem (see, for example, \cite[Theorem~5.4]{Mi2006}), which can be stated as follows.

\newtheorem*{theoremA}{Theorem A}
\begin{theoremA}
Suppose that $f$ is a holomorphic self-map of the open unit disc $\mathbb{D}$. Then either
\begin{enumerate}
\item $f$ is the identity function or an elliptic M\"obius transformation that fixes $\mathbb{D}$, or
\item there exists a point $\zeta \in\overline{\mathbb{D}}$ such that the sequence of iterates $f,f^2,f^3,\dotsc$ converges locally uniformly on $\mathbb{D}$ to $\zeta$. 
\end{enumerate}
\end{theoremA}

To explain the terminology in this theorem, an \emph{elliptic M\"obius transformation} that fixes $\mathbb{D}$ is a conformal automorphism of $\mathbb{D}$ that is conjugate by another conformal automorphism to a rotation about the origin. For each positive integer $n$, the $n$th \emph{iterate} $f^n$ of a holomorphic map $f$ is the function obtained by composing $f$ with itself $n$ times, $f^n=f\circ f\circ \dots \circ f$. The theorem states that the iterates $f,f^2,f^3,\dotsc$ converge locally uniformly on $\mathbb{D}$ to $\zeta$, meaning that the sequence of functions $f,f^2,f^3,\dotsc$ converges uniformly on compact subsets of $\mathbb{D}$ to $\zeta$, using the Euclidean metric on $\overline{\mathbb{D}}$. In case~(ii), the point $\zeta$ is called the \emph{Denjoy--Wolff point} of $f$; if $\zeta\in\mathbb{D}$ then it is a fixed point of $f$.

Our objective is to examine the stability of  Theorem~A under perturbations of the holomorphic map $f$, in a sense to be made precise shortly. We denote by $\mathcal{H}(\mathbb{D},\mathbb{C})$ the topological space of all holomorphic maps from $\mathbb{D}$ to the complex plane $\mathbb{C}$, equipped with the compact-open topology. In $\mathcal{H}(\mathbb{D},\mathbb{C})$, a sequence $(f_n)$ converges to a map $f$ if and only if $f_n\to f$ locally uniformly on $\mathbb{D}$. 

We focus on the subspace $\mathcal{H}(\mathbb{D})$ of holomorphic self-maps of $\mathbb{D}$. If $(f_n)$ is a sequence in $\mathcal{H}(\mathbb{D})$ that converges locally uniformly on $\mathbb{D}$ to a function $f$, then  either $f\in\mathcal{H}(\mathbb{D})$ or else $f$ is a constant function with value on the boundary of $\mathbb{D}$ (see \cite[Lemma~2.1]{Be2001b}). 

Given sequences $(f_n)$ and $(g_n)$ in $\mathcal{H}(\mathbb{D})$, we define the \emph{left-composition sequence} generated by $(f_n)$ and the \emph{right-composition sequence} generated by $(g_n)$ to be the sequences  
\[
F_n=f_n\circ f_{n-1} \circ \dots \circ f_1 \quad\text{and}\quad G_n=g_1\circ g_2\circ \dots\circ g_n,\quad n=1,2,\dotsc,
\]
respectively. Sequences of this type arise in a variety of contexts in dynamical systems, with differing notations and terminology. In future we omit the $\circ$ symbol from compositions.

The dynamical behaviour of the sequence of iterates $(f^n)$, where $f\in\mathcal{H}(\mathbb{D})$, depends on whether $f$ is the identity function, an elliptic M\"obius transformation, or if it has a Denjoy--Wolff point that lies in $\mathbb{D}$ or on the boundary of $\mathbb{D}$. We determine whether the dynamics of $(F_n)$ and $(G_n)$ are similar to that of $(f^n)$ under the assumptions that $f_n\to f$ and $g_n\to f$. We find that, in a sense, right-composition sequences are more stable than left-composition sequences when $f$ has a Denjoy--Wolff point inside $\mathbb{D}$, but the reverse holds when the Denjoy--Wolff point of $f$ lies on the boundary of $\mathbb{D}$. And when $f$ is the identity function, there is similar stability for both left- and right-composition sequences.

We make significant use of the hyperbolic metric on $\mathbb{D}$, which is the Riemannian metric $2|dz|/(1-|z|^2)$. We denote the corresponding distance function by $\rho$. Crucial to our study is the Schwarz--Pick lemma, which says that if $f\in\mathcal{H}(\mathbb{D})$, then $\rho(f(z),f(w))\leq \rho(z,w)$, for $z,w\in\mathbb{D}$, with equality if and only if $f$ is a conformal automorphism of $\mathbb{D}$. If $f$ is not a conformal automorphism, then for each compact subset $K$ of $\mathbb{D}$ we can find a positive constant $k<1$ such that $\rho(f(z),f(w))\leq k\rho(z,w)$, for $z,w\in K$.

There is an extensive literature on stability results for holomorphic dynamical systems; we draw attention to the papers of Beardon~\cite{Be2001a}, Gill~\cites{Gi1988,Gi1990} and Pommerenke~\cite{Po1981} for work closest to our own. Beardon and Gill were motivated in part by the theory of limit-periodic continued fractions, in which one considers the stability of continued fractions under perturbations of the coefficients. In \cite{Be2001a}, Beardon looks at the stability of M\"obius transformations under iteration. We develop the geometric approach of \cite{Be2001a}, and apply it to the class of holomorphic maps, which is far larger and more complex than the class of M\"obius transformations. Note that Theorem~\ref{thm3} of Section~\ref{sec3} could be deduced quickly from \cite[Theorem~4.7]{Be2001a} (the proof we give is short anyway). 

Gill studies composition sequences of holomorphic maps for which the constituent maps approach a limit function. Using Euclidean estimates he obtains results of a similar type to Theorems~\ref{thm3} and~\ref{thm4}. One of the benefits of our geometric approach is that we obtain strong results with succinct statements and concise proofs using the hyperbolic metric.

Pommerenke considers right-composition sequences $(F_n)$ under the assumption that $f_n\to f$, for some non-elliptic map $f$, and attempts to find constants $a_n$ and $b_n$ such that $a_nF_n+b_n\to F$, for some non-constant function $F$. Whether this is possible depends on the nature of the Denjoy--Wolff point of $f$. Our objectives are somewhat tangential to this, such that we obtain a complete analysis of stability for both left- and right- composition sequences and any choice of holomorphic map $f$.

\section{Stability at elliptic transformations and the identity function}

Here we consider the behaviour of the left- and right-composition sequences $F_n=f_nf_{n-1}\dotsb f_1$ and $G_n=g_1g_2\dotsb g_n$, where $f_n,g_n\in\mathcal{H}(\mathbb{D})$, under the assumption that the sequences $(f_n)$ and $(g_n)$ converge to an elliptic M\"obius transformation fixing $\mathbb{D}$ or the identity function $I$. We focus particularly on the latter case, because the iterates of an elliptic transformation do not themselves converge in $\mathcal{H}(\mathbb{D},\mathbb{C})$.

The next example demonstrates that when $f_n\to I$, and without further assumptions, the sequence $(F_n)$ can behave erratically.

\begin{example}\label{exa1} 
Let $f_n(z)=e^{i/n}z$, for $n=1,2,\dotsc$, so $f_n\to I$. Then 
\[
F_n(z) = e^{i\left(1+\tfrac12+\dots+\tfrac{1}{n}\right)}z.
\]
This sequence accumulates at the identity function and every rotation of the unit circle.\qed
\end{example}

Essentially the same example can be used with $g_n$ in place of $f_n$ and $G_n$ in place of $F_n$, because the functions commute.

We can get quite different behaviour with other choices for functions $f_n\to I$. For example, choosing $f_n(z)=(1-1/n)z$, for $n=2,3,\dotsc$, we see that $(F_n)$ converges locally uniformly on $\mathbb{D}$ to $0$.

Example~\ref{exa1} indicates that to obtain more controlled behaviour of $(F_n)$ and $(G_n)$ under the assumption that $f_n\to I$ and $g_n\to I$ we need additional constraints on convergence. Theorems~\ref{thm1} and~\ref{thm2}, to follow, show that such control can be achieved if we stipulate that the convergence is sufficiently fast (in a sense to be made precise). In fact, using the following result from \cite[Theorem 1.1]{ChSh2019}, we will see that it is sufficient to assume that $(f_n)$ and $(g_n)$ converge to the identity function suitably fast at just two  points in $\mathbb{D}$.

\newtheorem*{theoremB}{Theorem B}
\begin{theoremB}
Suppose that $f,g\in\mathcal{H}(\mathbb{D})$, with $g$ a conformal automorphism of $\mathbb{D}$,  and $a,b,z\in\mathbb{D}$, with $a\neq b$. Then
\[
\rho(f(z),g(z)) \leq \lambda\big(\rho(f(a),g(a))+\rho(f(b),g(b))\big),
\]
where
\[
\lambda= \frac{\exp\left(\rho(z,a)+\rho(a,b)+\rho(b,z)\right)}{\rho(a,b)}.
\]
\end{theoremB}

We now state our first result about stability of the Denjoy--Wolff theorem at the identity function or an elliptic transformation, for left-composition sequences.

\begin{theorem}\label{thm1}
Suppose that $f$ is either the identity function or an elliptic M\"obius transformation that fixes $\mathbb{D}$, and $f_1,f_2,\dotsc$ are non-constant holomorphic self-maps of $\mathbb{D}$ for which 
\[
\sum_{n=1}^\infty \rho(f_n(a),f(a))<+\infty\quad\text{and}\quad\sum_{n=1}^\infty \rho(f_n(b),f(b))<+\infty,
\]
for two distinct points $a,b\in\mathbb{D}$. Then the sequence $(f^{-n}F_n)$, where $F_n=f_n f_{n-1}\dotsb f_1$, converges locally uniformly on $\mathbb{D}$ to a non-constant holomorphic self-map of $\mathbb{D}$. 
\end{theorem}

\begin{proof}
Let $d = \tfrac13\rho(a,b)$ and let $K$ be a closed hyperbolic disc that is centred at a fixed point of $f$ and contains $a$ and $b$. Observe that if $z\in K$, then $f^n(z)\in K$, for $n\in\mathbb{Z}$.  By applying Theorem~B to the functions $f_n$ and $f$, for $n=1,2,\dotsc$, we see that
\[
\sum_{n=1}^\infty \sup_{z\in K} \rho(f_n(z),f(z))<+\infty.
\]
Notice that it suffices to prove the theorem for the truncated left-composition sequence with $n$th term $f_nf_{n-1}\dotsb f_N$, where $N$ is a fixed positive integer. In light of this observation, we may assume (after relabelling the functions) that in fact
\[
\sum_{n=1}^\infty\sup_{z\in K} \rho(f_n(z),f(z)) <d.
\]
Choose any point $z\in K$. Let $z_n=f^n(z)$, for $n=1,2,\dotsc$. Then $z_n\in K$. Observe that 
\begin{align*}
\rho(F_n(z),f^n(z)) &\leq \rho(f_n\dotsb f_1(z),f_n\dotsb f_2f(z))+\rho(f_n\dotsb f_2(f(z)),f^{n-1}(f(z)))\\
&\leq \rho(f_1(z),f(z))+\rho(f_n\dotsb f_2(z_1),f^{n-1}(z_1)),
\end{align*}
where, to obtain the second inequality, we have applied the Schwarz--Pick lemma with the function $f_n\dotsb f_2$. Repeating this argument we see that
\[
\rho(F_n(z),f^n(z)) \leq \rho(f_1(z),f(z))+\rho(f_2(z_1),f(z_1))+\dots+\rho(f_n(z_{n-1}),f(z_{n-1}))<d,
\]
for $n=1,2,\dotsc$.

Next, still with $z\in K$, we have
\begin{align*}
\rho(F_n(z),a) &\leq \rho(F_n(z),F_n(a))+\rho(F_n(a),f^n(a))+\rho(f^n(a),a)\\
&\leq \rho(z,a)+d+\rho(f^n(a),a)\leq l,
\end{align*}
for $n=1,2,\dotsc$, where $l$ is three times the hyperbolic diameter of $K$. Similarly $\rho(F_n(z),b)\leq l$. Applying Theorem~B to the functions $f_n$ and $f$, and with $F_{n-1}(z)$ in place of $z$, we obtain
\[
\rho(F_n(z),f(F_{n-1}(z))) \leq \lambda (\rho(f_n(a),f(a))+\rho(f_n(b),f(b))),
\]
where
\[
\lambda= \frac{\exp\left(\rho(F_{n-1}(z),a)+\rho(a,b)+\rho(b,F_{n-1}(z))\right)}{\rho(a,b)}\leq \frac{\exp(3l)}{\rho(a,b)}.
\]
Consequently, we see that
\[
\sum_{n=1}^\infty \rho(f^{-n}F_n(z),f^{-(n-1)}(F_{n-1}(z)))=\sum_{n=1}^\infty \rho(F_n(z),f(F_{n-1}(z))) <2\lambda d,
\]
for $z\in K$ (where $F_0$ is the identity function). Thus $(f^{-n}F_n)$ is a uniformly Cauchy sequence on $K$. Now, $K$ is an arbitrarily large compact subset of $\mathbb{D}$, so it follows that $(f^{-n}F_n)$ converges locally uniformly on $\mathbb{D}$ to a function $F$.

The function $F$ belongs to $\mathcal{H}(\mathbb{D})$, and it is not a constant function because
\[
\rho(f^{-n}F_n(a),f^{-n}F_n(b)) \geq \rho(a,b)-\rho(f^{-n}F_n(a),a)-\rho(f^{-n}F_n(b),b)>3d-d-d=d,
\]
for $n=1,2,\dotsc$. 
\end{proof}

When $f$ is the identity function $I$, Theorem~\ref{thm1} says that if $\sum \rho(f_n(a),a)<+\infty$ and $\sum \rho(f_n(b),b)<+\infty$, then the left-composition sequence $F_n=f_nf_{n-1}\dotsb f_1$ converges locally uniformly on $\mathbb{D}$ to a non-constant holomorphic map $F\in \mathcal{H}(\mathbb{D})$. And when $f$ is an elliptic transformation of finite order $m$, the theorem tells us that the sequence $(F_n)$ can be split into $m$ subsequences that converge to $F,fF,\dots,f^{m-1}F$, respectively. For the remaining case, when $f$ is an elliptic transformation of infinite order, we see from Theorem~\ref{thm1} that $(F_n)$ accumulates at uncountably many different non-constant maps in $\mathcal{H}(\mathbb{D})$.

Next we state a result similar to Theorem~\ref{thm1} for right-composition sequences.

\begin{theorem}\label{thm2}
Suppose that $g$ is either the identity function or an elliptic M\"obius transformation that fixes $\mathbb{D}$, and $g_1,g_2,\dotsc$ are non-constant holomorphic self-maps of $\mathbb{D}$ for which 
\[
\sum_{n=1}^\infty \rho(g_n(a),g(a))<+\infty\quad\text{and}\quad\sum_{n=1}^\infty \rho(g_n(b),g(b))<+\infty,
\]
for two distinct points $a,b\in\mathbb{D}$. Then the sequence $(G_ng^{-n})$, where $G_n=g_1 g_{2}\dotsb g_n$, converges locally uniformly on $\mathbb{D}$ to a non-constant holomorphic self-map of $\mathbb{D}$. 
\end{theorem}

\begin{proof}
Let  $d = \tfrac13\rho(a,b)$ and let  $K$  be a closed hyperbolic disc that is centred at a fixed point of $g$ and that contains $a$ and $b$. By truncating the right-composition sequence $(G_n)$ by a fixed finite number of terms from the left (and relabelling the remaining functions), we can assume that
\[
\sum_{n=1}^\infty\sup_{z\in K} \rho(g_n(z),g(z)) <d.
\]
Now choose a point $z$ in $K$, and let $n$ be a positive integer. By applying the Schwarz--Pick lemma with the function $G_{n-1}$, we see that
\[
\rho(G_ng^{-n}(z),G_{n-1}g^{-(n-1)}(z)) \leq \rho(g_n(w),g(w)),
\]
where $w=g^{-n}(z)$ (and $G_0$ is the identity function). Since $w\in K$, it follows that 
\[
\sum_{n=1}^\infty \rho(G_ng^{-n}(z),G_{n-1}g^{-(n-1)}(z)) < d.
\]
Therefore $(G_ng^{-n})$ is a uniformly Cauchy sequence on $K$, and since $K$ can be chosen to be arbitrarily large, we deduce that  $(G_ng^{-n})$ converges locally uniformly on $\mathbb{D}$ to a function $G$.

This function $G$ belongs to $\mathcal{H}(\mathbb{D})$; we must show that it is not a constant function. To this end, we write $a_n=g^{-n}(a)$, for $n=1,2,\dotsc$, and observe that
\begin{align*}
\rho(G_ng^{-n}(a),a) &\leq \rho(G_n(a_n),G_{n-1}(a_{n-1}))+\rho(G_{n-1}(a_{n-1}),G_{n-2}(a_{n-2}))+\dots +\rho(G_1(a_1),a)\\
&\leq \rho(g_n(a_n),g(a_n))+\rho(g_{n-1}(a_{n-1}),g(a_{n-1}))+\dots +\rho(g_1(a_1),g(a_1)),
\end{align*}
for $n=1,2,\dotsc$, where, to obtain the second inequality, we applied the Schwarz--Pick lemma with the functions $G_{n-1},G_{n-2},\dots G_0$, in that order. Since $a_n\in K$, for each index $n$, we find that $\rho(G_ng^{-n}(a),a)<d$, and similarly $\rho(G_ng^{-n}(b),b)<d$. Consequently,
\[
\rho(G_ng^{-n}(a),G_ng^{-n}(b)) \geq \rho(a,b)-\rho(G_ng^{-n}(a),a)-\rho(G_ng^{-n}(b),b)>3d-d-d=d,
\]
for $n=1,2,\dotsc$. Hence $G$ is a non-constant holomorphic self-map of $\mathbb{D}$.
\end{proof}

The special cases of Theorem~\ref{thm2} when the limit function $g$ is of finite order resemble the similar special cases of Theorem~\ref{thm1}. In particular, when $g$ is the identity function,  Theorem~\ref{thm2} says that if $\sum \rho(g_n(a),a)<+\infty$ and $\sum \rho(g_n(b),b)<+\infty$, then the right-composition sequence $G_n=g_1g_{2}\dotsb g_n$ converges locally uniformly on $\mathbb{D}$ to a non-constant holomorphic self-map of $\mathbb{D}$.

\section{Denjoy--Wolff point inside the disc}\label{sec3}

In this section we consider the stability of the Denjoy--Wolff theorem at holomorphic functions that have a Denjoy--Wolff point inside the unit disc. Central to our approach is the following theorem from \cite[Corollary 2.3]{BaRi1989} and \cite[Theorem 1.2]{Lo1990}.

\newtheorem*{theoremC}{Theorem C}
\begin{theoremC}
Suppose that $K$ is a compact subset of a simply connected hyperbolic domain $D$, and that $g_1,g_2,\dotsc$ are holomorphic maps of $D$ into $K$. Then the right-composition sequence $G_n=g_1g_2\dotsb g_n$ converges locally uniformly on $D$ to a constant in $K$.
\end{theoremC}

Using Theorem~C we obtain the following strong stability result for right-composition sequences.

\begin{theorem}\label{thm3}
Let $g$ be a holomorphic self-map of $\mathbb{D}$ with a Denjoy--Wolff point $\zeta$ in $\mathbb{D}$. Then there is a neighbourhood $\mathcal{U}$ of $g$ in $\mathcal{H}(\mathbb{D})$ such that if $g_1,g_2,\dotsc$ belong to $\mathcal{U}$, then the right-composition sequence $G_n=g_1g_{2}\dotsb g_n$ converges locally uniformly on $\mathbb{D}$ to a constant in $\mathbb{D}$.
\end{theorem}

We use the notation $D(c,r)$ for the hyperbolic open disc with centre $c$ and radius $r$.

\begin{proof}
Let $D=D(\zeta,r)$, for some $r>0$. Since $\overline{D}$ is a compact set in $\mathbb{D}$, we see from the Schwarz--Pick lemma that there is a positive constant $k<1$ (that depends on $\overline{D}$) with $\rho(g(z),g(w))\leq k\rho(z,w)$, for $z,w\in \overline{D}$. Observe that $g$ fixes $\zeta$, so  $g(D)\subset D(\zeta,s)$, where $s=kr$. Now choose a real number $t$ with $s<t<r$.  Let
\[
\mathcal{U} = \{h\in\mathcal{H}(\mathbb{D}): h(D) \subset D(\zeta,t)\},
\]
a neighbourhood of $g$ in $\mathcal{H}(\mathbb{D})$, and let $K=\overline{D(\zeta,t)}$. If $g_1,g_2,\dotsc$ belong to $\mathcal{U}$, then $g_n(D) \subset K$, for each index $n$, so we can apply Theorem~C to see that the right-composition sequence $G_n=g_1g_2\dotsb g_n$ converges locally uniformly on $D$ to a constant in $K$. And, since the radius $r$ of $D$ was chosen arbitrarily, it follows that $(G_n)$ converges locally uniformly on $\mathbb{D}$ to a constant in $\mathbb{D}$.
 \end{proof}

The hypotheses of Theorem~\ref{thm3} can of course be weakened to assume that all but finitely many of the maps $g_n$ belong to $\mathcal{U}$.

The next example shows that there is no analogue of Theorem~\ref{thm3} for left-composition sequences.

\begin{example}\label{exa2} 
Let $f(z)=z/2$, and let $\mathcal{U}$ be a neighbourhood of $f$ in $\mathcal{H}(\mathbb{D})$. We can choose a positive constant $\delta$ sufficiently small that all the functions $f_n(z)=z/2+\delta e^{i\theta_n}$, where $\theta_n\in\mathbb{R}$, for $n=1,2,\dotsc$, belong to $\mathcal{U}$. The left-composition sequence $F_n=f_nf_{n-1}\dotsb f_1$ satisfies
\[
F_n(z) = \tfrac12 F_{n-1}(z) +\delta e^{i\theta_n}.
\]
Evidently, the parameters $\theta_n$ can be chosen so that $(F_n)$ diverges pointwise on $\mathbb{D}$.\qed
\end{example}

With slightly stronger hypotheses, however, we do obtain controlled behaviour of the left-composition sequence $(F_n)$.

\begin{theorem}\label{thm4}
Let $f$ be a holomorphic self-map of $\mathbb{D}$ with a Denjoy--Wolff point $\zeta$ in $\mathbb{D}$. Suppose that $f_1,f_2,\dotsc$ is a sequence of functions in $\mathcal{H}(\mathbb{D})$ that converges locally uniformly on $\mathbb{D}$ to $f$. Then the left-composition sequence $F_n=f_nf_{n-1}\dotsb f_1$ converges locally uniformly on $\mathbb{D}$ to $\zeta$.
\end{theorem}

\begin{proof}
Let $K$ be a closed hyperbolic disc centred at $\zeta$. Observe that $f$ maps $K$ inside a smaller closed hyperbolic disc centred at $\zeta$. Since $f_n\to f$ uniformly on $K$ we see that $f_n$ maps $K$ inside itself for sufficiently large $n$. By truncating $F_n$ by finitely many terms on the right (and relabelling) we can assume that in fact $f_n(K)\subset K$ for all $n=1,2,\dotsc$. 

We define $k$ to be a constant between $0$ and $1$ for which $\rho(f(z),f(w))\leq k\rho(z,w)$, for $z,w\in K$.

Choose $z\in K$. Observe that $f^n(z)\in K$ and $F_n(z)\in K$, for $n=1,2,\dotsc$. Then
\begin{align*}
\rho(F_n(z),f^n(z)) &\leq \rho(F_n(z),f(F_{n-1}(z)))+\rho(f(F_{n-1}(z)),f^n(z))\\
&\leq \sup_{w\in K}\rho (f_n(w),f(w)) + k\rho(F_{n-1}(z),f^{n-1}(z)),
\end{align*}
for $n=1,2,\dotsc$. Repeating this argument, we see that 
\[
\rho(F_n(z),f^n(z)) \leq (1+k+k^2+\dots+k^{n-1})\sup_{w\in K}\rho (f_n(w),f(w))\leq \frac{1}{1-k}\sup_{w\in K}\rho (f_n(w),f(w)),
\] 
for $n=1,2,\dotsc$. Since $(f_n)$ converges locally uniformly on $\mathbb{D}$ to $f$ we see that $\rho(F_n(z),f^n(z))\to 0$ uniformly on $K$, so $F_n\to \zeta$ uniformly on $K$. Hence $(F_n)$ converges locally uniformly on $\mathbb{D}$ to the constant $\zeta$.
\end{proof}

Notice that the left-composition sequence $(F_n)$ of Theorem~\ref{thm4} converges locally uniformly on $\mathbb{D}$ to $\zeta$, but the right-composition sequence $(G_n)$ of  Theorem~\ref{thm3} converges to a constant that need not be $\zeta$. After all, adjusting $g_1$ causes the constant to change.

\section{Denjoy--Wolff point on the boundary of the disc}

This final section considers the stability of the Denjoy--Wolff theorem at holomorphic functions $f$ that have a Denjoy--Wolff point on the boundary of the unit disc. In a sense, this circumstance is the least stable of those considered so far. Indeed, it is straightforward to find holomorphic maps $f_1,f_2,\dotsc$ with $f_n\to f$ (for a suitable choice of $f$ with a Denjoy--Wolff point on the boundary of $\mathbb{D}$) for which the behaviour of the left-composition sequence $F_n=f_nf_{n-1}\dotsb f_1$ is erratic. Nevertheless, the following theorem shows that if we assume that the convergence of $(f_n)$ to $f$ is sufficiently rapid, then the sequences $(F_n)$ and $(f^n)$ have similar dynamics. 

\begin{theorem}\label{thm5}
Let $f$ be a holomorphic self-map of $\mathbb{D}$ with a Denjoy--Wolff point $\zeta$ on the boundary of $\mathbb{D}$. Then there exist neighbourhoods $\mathcal{U}_1, \mathcal{U}_2,\dotsc$ of $f$ in $\mathcal{H}(\mathbb{D})$ such that if $f_n\in \mathcal{U}_n$, for $n=1,2,\dotsc$, then the left-composition sequence $F_n=f_n f_{n-1}\dotsb f_1$ converges locally uniformly on $\mathbb{D}$ to $\zeta$.
\end{theorem}

\begin{proof}
For each positive integer $n$, we define $D_n$ to be the open hyperbolic disc centred at $0$ of radius $1+\rho(f^{n-1}(0),0)$, and let
\[
\mathcal{U}_n = \{h \in \mathcal{H}(\mathbb{D}): \rho(h(z),f(z))<1/2^n\text{ for }z\in D_n\},
\]
a neighbourhood of $f$ in $\mathcal{H}(\mathbb{D})$. Suppose that $f_n\in \mathcal{U}_n$, for $n=1,2,\dotsc$.

We will prove by induction on $m$ that
\[
\rho(F_m(0),f^m(0)) < 1 - \frac{1}{2^m},
\]
for $m=1,2,\dotsc$. This is certainly true for $m=1$, by definition of $\mathcal{U}_1$. Suppose that it is true for the integer $m=n-1$, where $n>1$. Then  
\begin{align*}
\rho(F_n(0),f^n(0)) &\leq \rho(F_n(0),f(F_{n-1}(0)))+\rho(f(F_{n-1}(0)),f^n(0))\\
&\leq \rho(F_n(0),f(F_{n-1}(0)))+\rho(F_{n-1}(0),f^{n-1}(0))\\
&<\rho(F_n(0),f(F_{n-1}(0)))+1-\frac{1}{2^{n-1}},
\end{align*}
where we have applied the triangle inequality, the Schwarz--Pick lemma, and the induction hypothesis. Now, since 
\[
\rho(F_{n-1}(0),0) \leq \rho(F_{n-1}(0),f^{n-1}(0))+\rho(f^{n-1}(0),0)<1 + \rho(f^{n-1}(0),0),
\]
we see that $F_{n-1}(0)\in D_{n}$. So, by definition of $\mathcal{U}_n$, we have
\[
\rho(F_n(0),f(F_{n-1}(0)))=\rho(f_n(F_{n-1}(0)),f(F_{n-1}(0)))<\frac{1}{2^n}.
\] 
Combining the inequalities obtained we conclude that
\[
\rho(F_n(0),f^n(0))<\rho(F_n(0),f(F_{n-1}(0)))+1-\frac{1}{2^{n-1}}< \frac{1}{2^n}+1-\frac{1}{2^{n-1}}=1 - \frac{1}{2^n}.
\] 
This completes the proof by induction.

A consequence of this observation is that $\rho(F_n(0),f^n(0))<1$, for each positive integer $n$. Then, since $f^n(0)\to \zeta$, a point on the boundary of $\mathbb{D}$, we can use a formula for the hyperbolic metric in $\mathbb{D}$ such as 
\[
\sinh \tfrac12 \rho (z,w) = \frac{|z-w|}{\sqrt{(1-|z|^2)(1-|w|^2)}},
\]
to see that $F_n(0)\to \zeta$ also.

Furthermore, we have that $\rho(F_n(z),F_n(0))\leq \rho(z,0)$, for any point $z\in\mathbb{D}$, and from this inequality we see that $(F_n)$ converges locally uniformly on $\mathbb{D}$ to $\zeta$ (with convergence in the Euclidean metric).
\end{proof}

There is no such result as Theorem~\ref{thm5} for right-composition sequences. To see this, consider the function $g(z)=z+1$ acting on the upper half-plane $\mathbb{H}$ with Denjoy--Wolff point $\infty$. (Here $\mathbb{H}$ takes the place of the unit disc $\mathbb{D}$.) Let $h(z)=i+e^{2\pi iz}$, which is a holomorphic self-map of $\mathbb{H}$ that satisfies $hg=h$. Now consider the right-composition sequence $G_n=g_1g_2\dotsb g_n$, where $g_1=h$ and $g_n=g$, for $n>1$. Then $(g_n)$ converges to $g$ in the fastest possible way, but $G_n=hg^{n-1}=h$.

The following, similar example exhibits even worse behaviour of the sequence $(G_n)$. We provide only a sketch of the details, which requires the theory of prime ends (see, for example, \cite[Section~17]{Mi2006}).

\begin{example}\label{exa3} 
This example also uses $\mathbb{H}$ rather than $\mathbb{D}$. We define $g(z)=z/2$, which is a holomorphic self-map of $\mathbb{H}$ with Denjoy--Wolff point $0$. Let $D$ be the simply connected domain shown in Figure~\ref{fig1}. It is obtained by removing two vertical line segments and various horizontal line segments from $\mathbb{H}$ to leave an infinite snake-like domain, as shown in the figure. There are infinitely many horizontal line segments, and they accumulate at the real interval $[-1,1]$, which is a prime end of $D$.

\begin{figure}[ht]
\centering
\begin{tikzpicture}[scale=1.91]
\newcommand{\eps}{0.2}
\draw (-2,0) --(2,0);
\draw (-1,0)--  (-1,1);
\draw (1,0)--  (1,1);

\draw (-1,1) -- (1-\eps,1);
\draw (-1+\eps,0.5) -- (1,0.5);
\draw (-1,0.3) -- (1-\eps,0.3);
\draw (-1+\eps,0.2) -- (1,0.2);
\draw (-1,0.13) -- (1-\eps,0.13);
\draw (-1+\eps,0.08) -- (1,0.08);
\draw (-1,0.04) -- (1-\eps,0.04);
\draw (-1+\eps,0.02) -- (1,0.02);

\node [below] at (0,0) {0};
\node [below] at (-1,0) {$-1\phantom{-}$};
\node [below] at (1,0) {$1$};
\end{tikzpicture}
\caption{Domain $D$ with a prime end at $[-1,1]$} 
\label{fig1}
\end{figure}
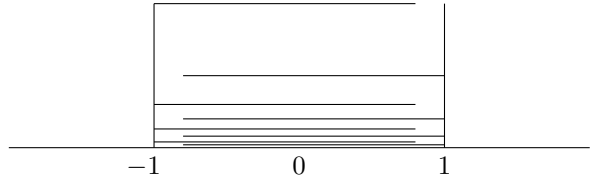

We define $h$ to be a conformal map from $\mathbb{H}$ to $D$. This map induces a one-to-one correspondence between the extended real line (the boundary of $\mathbb{H}$) and the prime ends of $D$. We choose $h$ such that $0$ corresponds to the prime end $[-1,1]$. Now consider the right-composition sequence $G_n=g_1g_2\dotsb g_n$, where $g_1=h$ and $g_n=g$, for $n>1$. Then $(g_n)$ converges to $g$ as quickly as possible, however, we will show that $(G_n(i))$ diverges. To see this, first observe that
\[
G_n(i) = hg^n(i) = h(i/2^n), \quad \text{for $n=1,2,\dotsc$.}
\]
Since $h$ is a conformal map from $\mathbb{H}$ to $D$, it preserves hyperbolic distance between these two domains. So the hyperbolic length of the hyperbolic geodesic $\Gamma_n$ between $G_{n-1}(i)$ and $G_n(i)$ in $D$ is equal to the hyperbolic distance between $i/2^{n-1}$ and $i/2^n$ in $\mathbb{H}$, namely $\log 2$. Now, as $n$ increases, $i/2^n$ approaches $0$ (in the Euclidean metric), and $G_n(i)$ approaches the prime end $[-1,1]$ (in the Euclidean metric). By applying a simple estimate with the quasihyperbolic metric, it can then be shown that the Euclidean length of $\Gamma_n$ converges to $0$. And from the shape of $D$ we can see that $(G_n(i))$ accumulates at an interval within $[-1,1]$, so it diverges.\qed
\end{example}

Example~\ref{exa3} indicates that there is little hope of obtaining a simple analogue of Theorem~\ref{thm5} for right-composition sequences. It also suggests that we ought to shift our perspective when considering right-composition sequences, in the following sense. The sequence $(G_n(i))$ certainly diverges in the closure of the domain $\mathbb{H}$, but it converges in the Carath\'eodory compactification of the domain $D$, to the prime end $[-1,1]$. In general, for a right-composition sequence $G_n=g_1g_2\dotsb g_n$ acting on $\mathbb{D}$, it is likely to be more rewarding to consider convergence of $(G_n)$ not with respect to $\mathbb{D}$, but with respect to the set $\bigcap G_n(\mathbb{D})$ (or perhaps its interior), which in many cases will be a simply connected domain. We will examine this idea more thoroughly in future work.

\end{document}